\documentclass[12pt]{article}

\usepackage[usenames]{color}
\usepackage{amssymb}
\usepackage{amsmath}
\usepackage{amsthm}
\usepackage{amsfonts}
\usepackage{amscd}
\usepackage{graphicx}

\usepackage[colorlinks=true,
linkcolor=webgreen,
filecolor=webbrown,
citecolor=webgreen]{hyperref}

\definecolor{webgreen}{rgb}{0,.5,0}
\definecolor{webbrown}{rgb}{.6,0,0}

\usepackage{color}
\usepackage{fullpage}
\usepackage{float}

\usepackage{graphics}
\usepackage{latexsym}
\usepackage{epsf}
\usepackage{breakurl}

\newcommand{\seqnum}[1]{\href{https://oeis.org/#1}{\rm \underline{#1}}}

\begin{document}

\theoremstyle{plain}
\newtheorem{theorem}{Theorem}
\newtheorem{corollary}[theorem]{Corollary}
\newtheorem{lemma}[theorem]{Lemma}
\newtheorem{proposition}[theorem]{Proposition}

\theoremstyle{definition}
\newtheorem{definition}[theorem]{Definition}
\newtheorem{example}[theorem]{Example}
\newtheorem{conjecture}[theorem]{Conjecture}
\newtheorem{question}[theorem]{Question}

\theoremstyle{remark}
\newtheorem{remark}[theorem]{Remark}

\begin{center}
\vskip 1cm{\LARGE\bf Bubbles in Linear Chord Diagrams: Bridges
  and Crystallized Diagrams}
\vskip 1cm 
\large
Donovan Young\\
St Albans, Hertfordshire, AL3 4HB\\
United Kingdom\\
\href{mailto:donovan.m.young@gmail.com}{\tt donovan.m.young@gmail.com} \\
\end{center}

\vskip .2 in

\begin{abstract}
In a linear chord diagram a short chord joins adjacent vertices while
a bubble is a region devoid of short chords. We define a bridge to be
a chord joining a vertex interior to a bubble to one exterior to
it. Building on earlier work, we investigate the distribution of
bridges in the limit of large bubbles and diagrams, and show that the
number of bridges is asymptotically normal, obtaining expressions for
the associated mean and variance as a function of bubble size. We
introduce the notion of a crystallized diagram, defined by the
criteria that all its chords are either short or are bridges. We count
the number of crystallized diagrams by the number of short chords they
contain, and provide the asymptotic distribution in the limit of large
crystallized diagrams. We show that for very large diagrams, the
number of short chords is normal, and sharply peaked at $\sqrt{2n/\log
  n}$, where $n$ is the total number of chords in the diagram.
\end{abstract}

\section{Introduction}

A linear chord diagram is a linear arrangement of $2n$ vertices
connected in pairs by $n$ unoriented arcs called {\it chords}, so that
each vertex participates in exactly one chord. Elementary
considerations reveal that there are $(2n-1)!!=(2n-1)(2n-3)\cdots 1$
distinct linear chord diagrams consisting of $n$ chords.

One interesting way of refining this counting\footnote{The
  combinatorics of linear chord diagrams has a long history beginning
  with Touchard \cite{T} and Riordan's \cite{R} studies of the number
  of chord crossings; cf.\ Pilaud and Ru\'{e} \cite{PR} for a modern
  approach and further developments. Krasko and Omelchenko \cite{KO}
  provide a more complete list of references.} is by the number of
so-called {\it short} chords, i.e., chords that join adjacent
vertices. Kreweras and Poupard~\cite{KP} provided recurrence relations
and closed form expressions for the number of diagrams with exactly
$k$ short chords. They also showed that the mean number of short
chords is $1$, which implies that the total number of short chords is
equinumerous with the total number of linear chord diagrams,
cf.\ \cite{CK}. Kreweras and Poupard~\cite{KP} showed further that all
higher factorial moments of the distribution approach $1$ in the
$n\to\infty$ limit, thus establishing the Poisson nature of the
asymptotic distribution.

In a previous publication the present author~\cite{DY4} defined the
concept of a {\it bubble} in a linear chord diagram. This is a
non-empty set of adjacent vertices, bounded either by the ends of the
diagram or by short chords, which itself is devoid of short chords,
see Figure \ref{n2}. In that paper, bubbles were enumerated by their
size (i.e. by the number of vertices they contain), and the form of
the asymptotic distribution of these sizes was determined.

In the present paper we will refine this counting by the number of
chords which connect vertices inside a bubble to regions outside of
it. We call these chords {\it bridges}. We may count bubbles by the
number of bridges contained within them. Let $N_{q,b}(n)$ be the
number of bubbles of size $q\in [1,2n]$, with exactly $b \in [0,q]$
bridges, counted across all diagrams consisting of $n$ chords. For
fixed $n$, the maximum value of $b$ occurs when the diagram contains a
single, central short chord and hence two bubbles of size $n-1$
connected by $n-1$ bridges (cf.\ the middle diagram in Figure
\ref{n2}). We therefore have that $N_{q,b}(n)$ is a $2n \times n$
array. For example, we have that:
\begin{equation}\nonumber
N_{q,b}(2)=
\begin{pmatrix}
0 &2 \\ 0 &0 \\ 0 &0 \\ 1 &0 \\
\end{pmatrix}.
\end{equation}
In order to see this counting, consider the three diagrams consisting
of $2$ chords, as pictured in Figure \ref{n2}.
\begin{figure}[H]
\begin{center}
\includegraphics[width=4.in]{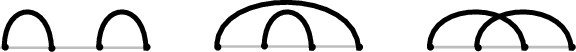}
\end{center}
\caption{The three linear chord diagrams built from $2$ chords. The
  diagram on the left has no bubbles. The diagram in the middle has
  two bubbles of size $1$, both of which contain a (common)
  bridge. The diagram on the right is itself a bubble of size $4$ and
  therefore contains no bridges.}
\label{n2}
\end{figure}
\noindent The diagram which is itself a bubble of size $4$ clearly has
no bridges, hence $N_{4,0}(2)=1$. The only other diagram which has
bubbles has a short chord across the two central vertices, as
previously mentioned, and two bubbles of size $1$ to either side, each
of which consists of a single vertex connected by a bridge. We
therefore have that $N_{1,1}(2)=2$. The counts for the next few values
of $n$ are as follows:
\begin{equation}\nonumber
N_{q,b}(3)=
\begin{pmatrix}
0 &8 &0\\ 0 &0 &4\\ 0 &2 &0\\ 2 &0 &0\\ 0 &0 &0\\ 5 &0
&0\\
\end{pmatrix},
N_{q,b}(4)=
\begin{pmatrix}
0 &42 &0 &0\\
0 &0 &30 &0\\
0 &8 &0 &12\\
3 &0 &12 &0\\
0 &12 &0 &0\\
10 &0 &0 &0\\
0 &0 &0 &0\\
36 &0 &0 &0\\
\end{pmatrix},
N_{q,b}(5)=
\begin{pmatrix}
0 &300 &0 &0 &0\\

0 &0 &240 &0 &0\\

0 &42 &0 &144 &0\\

9 &0 &90 &0 &48\\

0 &48 &0 &72 &0\\

15 &0 &84 &0 &0\\

0 &82 &0 &0 &0\\

72 &0 &0 &0 &0\\

0 &0 &0 &0 &0\\

329 &0 &0 &0 &0\\
\end{pmatrix}.
\end{equation}
There are a few general remarks to be made about these numbers. First
of all, we note that a bubble of size $1$ has exactly one bridge, and
so the first row is always zero except in the second entry. As a
bubble's size increases from $1$, its capacity to support bridges
increases as well. At a certain point, however, as a bubble becomes
large, it occupies more of the diagram's available vertices, and
thereby decreases the capacity for chords which bridge the interior
and exterior of the bubble, because there are fewer exterior vertices
available. The precise relationship between the mean number of bridges
and the size of the bubble, in the limit of large diagrams and
bubbles, will be found below.

The sum of the $q^\text{th}$ row is the total number of bubbles of
size $q$, given by \seqnum{A367000}. A diagram without short chords
can be regarded as a bubble of size $2n$ and clearly has no
bridges. For this reason the last row consists of zeroes except for
the first entry, which is the corresponding term of \seqnum{A278990},
the number of diagrams without short chords. A bubble of size $2n-1$
is impossible, as a short chord occupies two vertices, hence the
penultimate row is always a row of zeroes.

\begin{theorem}\label{ThmBbar}
  Let $\bar b$ denote the mean number of bridges counted across all
  bubbles of size $q$ found in linear chord diagrams with $n$
  chords. Then
  \begin{equation}\label{bbar}
    \lim_{n,q\to\infty}\bar b = \frac{q(2n-q)}{2n}.
  \end{equation}
\end{theorem}

We will provide a proof of this theorem below, but first we will
present a heuristic argument. Consider a bubble $B$ of size
$q=\mathcal{O}(n)$ bounded by a short chord and consider the vertex
$v_1$ immediately outside the bubble, see Figure \ref{FigIntuit}. We
then increase the bubble size by $1$ by exchanging the positions of
the short chord with $v_1$, so that $v_1$ is subsumed into the bubble.
\begin{figure}[H]
\begin{center}
\includegraphics[width=4.in]{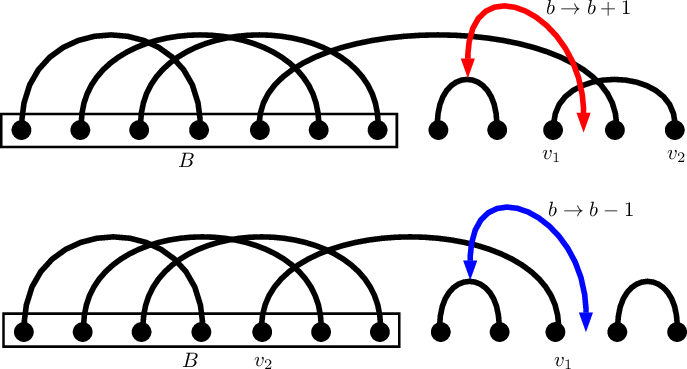}
\end{center}
\caption{When a short chord bounding a bubble $B$ exchanges its
  position with the exterior vertex $v_1$ which is adjacent to it, the
  number of bridges in $B$ either increases (or decreases) by one. The
  former case is pictured on top and occurs when the vertex $v_2$
  (joined to $v_1$ via a chord) is exterior to $B$. The latter case is
  pictured below and occurs when $v_2$ is within $B$.}
\label{FigIntuit}
\end{figure}
\noindent We now consider the vertex
$v_2$ which is connected to $v_1$ via a chord. If $v_2$ is within the
original bubble, the exchange reduces the number of bridges by one,
whereas if $v_2$ is outside the bubble the exchange increases the
number of bridges by one. We therefore have that
\begin{equation}\nonumber
  \frac{\Delta \bar b}{\Delta q} \simeq P(v_2\notin B) - P(v_2\in B)
  \simeq \frac{2n-q}{2n} - \frac{q}{2n},
\end{equation}
where we have used the fact that for large $n$, the position of $v_2$
is approximately uniformly distributed. Integrating this result, and
applying the boundary condition that a bubble of size $2n$ has exactly
zero bridges, we obtain Equation (\ref{bbar}).

\section{Asymptotics}

We introduce a model for approximating $N_{q,b}(n)$, valid for bubbles
of size $q={\cal O}(n)$, and with $b={\cal O}(n)$ bridges, as $n$ is
taken to infinity.
\begin{equation}\label{modeleq}
\begin{split}
N_{q,b}(n)\simeq
  2&\sum_{p=0}^b
    \frac{\left(2(n-1)-q\right)_b
 \left(2(n-1)-(q+b)-1\right)!!
      (q-b-1)!! 
   }{ep!} \\
   & \qquad \times\sum_{b_0=0}^{b-p}
{b_0+q-b-p\choose b_0}
{b-b_0-1\choose p-1}.
  \end{split}
\end{equation}
We explain each of the factors in this expression in turn, starting
from the left and moving right.
\begin{figure}[H]
\begin{center}
\includegraphics[width=4.in]{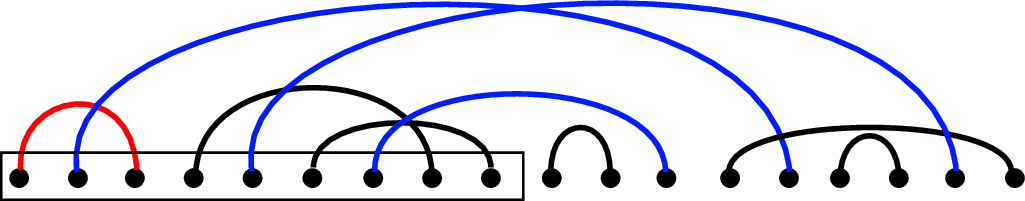}
\end{center}
\caption{A specific diagram of the class whose count is approximated
  by the model given in Eq. (\ref{modeleq}), with $n=9$. The bubble
under consideration is indicated by the boxed vertices and has size
$q=9$. There are $b=3$ bridges emanating from the bubble, indicated in
blue. The bubble has $p=1$ would-be short chords (indicated in red),
and since this chord contains one of the bridges' endpoints, there are
$b_0=2$ remaining endpoints within the bubble.}
\label{intuit}
\end{figure}
\begin{enumerate}

\item We take a bubble of size $q$ to be bounded by a single short
  chord, and one end of the diagram. We will explain why this captures
  the main contribution to $N_{q,b}(n)$ below this list. There is
  therefore a factor of $2$ to count the two ends.
  
\item\label{tr} Each of the $b$ bridges may terminate outside the
  bubble on any of the $2(n-1)-q$ external vertices. We will treat the
  bridges as indistinguishable in what follows, and so count them here
  as distinguishable: there are therefore $(2(n-1)-q)_b =
  (2(n-1)-q)!/(2(n-1)-q-b)!$ ways of placing them.

\item\label{tq} The $2(n-1)-q-b$ vertices outside the bubble, not
  participating in any of the $b$ bridges, may be connected in any of
  $(2(n-1)-q-b-1)!!$ ways.

\item The $q-b$ vertices inside the bubble, not participating in any
  of the $b$ bridges, may be connected in any of $(q-b-1)!!$ ways. If
  we were to delete the $b$ internal vertices participating in the
  bridges, then amongst these $(q-b-1)!!$ configurations, according to
  the work of Kreweras and Poupard \cite{KP}, approximately
  $e^{-1}(q-b-1)!!/p!$ of them would contain $p$ short chords. Note
  that we may take $p={\cal O}(1)$ as the factor $1/p!$ suppresses
  large values of $p$.

\item Of the $b$ vertices within the bubble participating in bridges,
  let $b_0$ of them lie outside any of the $p$ would-be short
  chords. The enumeration of placements of the $b_0$ vertices is then
  the problem of placing $b_0$ indistinguishable balls into $q-b-p+1$
  distinguishable bins (the spaces between the vertices). This
  accounts for the factor of ${b_0+q-b-p\choose b_0}$.

\item The remaining $b-b_0$ vertices must be placed within the $p$
  would-be short chords. The enumeration here is the placement of
  $b-b_0$ indistinguishable balls into $p$ distinguishable bins, so
  that no bin is empty. This accounts for the factor of
  ${b-b_0-1\choose p-1}$.
\end{enumerate}
If we were to include contributions from bubbles bounded by two short
chords, then the factors in items \ref{tr} and \ref{tq} would have $n$
replaced with $n-1$, bringing about a relative suppression of ${\cal
  O}(n^{-2})$. Although we would then also sum over the position of
the bubble, this would only provide an additional factor of ${\cal
  O}(n)$.

Strictly speaking, we should impose the constraints $b\geq p$, $q \geq
3p$, as all $p$ would-be short chords must be punctured by at least
one bridge each. Since we take $b$ and $q$ to be ${\cal O}(n)$ and $p$
to be ${\cal O}(1)$, we assume these constraints to be satisfied. In
Equation (\ref{modeleq}), the sum over $b_0$ may be performed to give
\begin{equation}\nonumber
N_{q,b}(n)\simeq 2
    \frac{\left(2(n-1)-q\right)_b
 \left(2(n-1)-(q+b)-1\right)!!
      (q-b-1)!! 
   }{e} \sum_{p=0}^b\frac{1}{p!}{q-p\choose b-p},
\end{equation}
but since $q$ and $b$ are ${\cal O}(n)$ and $p$ is ${\cal O}(1)$, we have
\begin{equation}\nonumber
N_{q,b}(n)\simeq 2
    \frac{\left(2(n-1)-q\right)_b
 \left(2(n-1)-(q+b)-1\right)!!
      (q-b-1)!! 
    }{e} {q\choose b}
    \sum_{p=0}^b  \frac{1}{p!}\left(1 + {\cal O}(n^{-1}) \right).
\end{equation}
Concentrating on the leading contribution, and approximating $\sum_p
1/p! \simeq e$, we have
\begin{equation}\nonumber
N_{q,b}(n)\simeq
  2\frac{(2(n-1)-q)!}{(2(n-1)-q-b)!}\frac{(2(n-1)-(q+b))!}{2^{n-1-(q+b)/2}
    (n-1-(q+b)/2)!}
  \frac{(q-b)!}{2^{(q-b)/2}
    ((q-b)/2)!}
  \frac{q!}{b!(q-b)!},
\end{equation}
where we have used $(2n-1)!!=(2n)!/(2^nn!)$. Simplifying further we find
\begin{equation}\nonumber
N_{q,b}(n)\simeq
  \frac{1}{2^{n-1-(q+b)/2}
    (n-1-(q+b)/2)!}
  \frac{2(2(n-1)-q)!}{2^{(q-b)/2}
    ((q-b)/2)!}
  \frac{q!}{b!}.
\end{equation}
We now use this expression to provide a proof of Theorem \ref{ThmBbar}.
\begin{proof}
We will show that the number of bridges $b$ is asymptotically normal
with a mean given by Equation (\ref{bbar}) and a variance which we
will presently determine. We employ the Stirling approximation $N!
\simeq \sqrt{2\pi N}(N/e)^N$, and take a logarithm of $N_{q,b}(n)$. We
will only be interested in those terms dependent on $b$, with all
others being treated as constants:
\begin{equation}\nonumber
  \begin{split}
\log N_{q,b}(n)\simeq
    &-(n-1-(q+b)/2)\log 2-((q-b)/2)\log 2\\
    &- (n-1-(q+b)/2 + 1/2)\left(\log(n-1-(q+b)/2)-1\right)
  \\
    &-((q-b)/2+1/2)\left(\log((q-b)/2)-1\right)
    -(b+1/2)\left(\log b-1\right) + \text{const.}
    \end{split}
\end{equation}
The mean value of $b$ is found by imposing $\frac{\partial}{\partial
  b}\log N_{q,b}(n)=0$. We have that
\begin{equation}\nonumber
  \begin{split}
&\frac{\partial}{\partial b}\log N_{q,b}(n)\simeq
    \frac{1}{2}\log 2
    +\frac{1}{2}\left(\log(n-1-(q+b)/2)-1\right) +\frac{1}{2}\\
  &+\frac{1}{2}\log 2
    +\frac{1}{2}\left(\log((q-b)/2)-1\right) +\frac{1}{2}
    -\left(\log b-1\right) -1 + {\cal O}(n^{-1})\\
    &=\frac{1}{2}\log\left(\frac{\left(2(n-1)-(q+b)\right)
      \left(q-b\right)}{b^2}\right)+ {\cal O}(n^{-1}),
  \end{split}
\end{equation}
and so the leading contribution to the mean value of $b$ is found by imposing
\begin{equation}\nonumber
\frac{\left(2(n-1)-(q+b)\right)
      \left(q-b\right)}{b^2}=1.
\end{equation}
Solving for $b$ we find
\begin{equation}\label{EqBbar}
\bar b = \frac{q\left(2(n-1)-q\right)}{2(n-1)},
\end{equation}
which is consistent with Equation (\ref{bbar}) in the $n\to\infty$
limit. It is also ${\cal O}(n)$ which is consistent with our
approximations. The variance of $b$ can be found by computing
\begin{equation}
  \begin{split}
  \left(
    -\frac{\partial^2}{\partial
      b^2}\log N_{q,\bar b}(n)\right)^{-1}
    &=\left(\frac{1}{2\left(2(n-1)+(q+\bar b)\right)}
    +\frac{1}{2\left(q-\bar b\right)}+\frac{1}{\bar b}\right)^{-1}\\
    &=\frac{q^2\left(2(n-1)-q\right)^2}{4(n-1)^3},\label{EqVarb}
    \end{split}
\end{equation}
which tells us that the standard deviation of $b$ is ${\cal
  O}(\sqrt{n})$. Generically one finds that higher derivatives are
suppressed as follows
\begin{equation}\nonumber
\frac{\partial^p}{\partial
      b^p}\log N_{q,\bar b}(n)  = {\cal O}(n^{-p+1}),
\end{equation}
and since this decreases faster than $\sqrt{\text{Var}(b)}^{-p}={\cal
  O}(n^{-p/2})$, normality follows in the $n\to\infty$ limit.
\end{proof}
In Figure \ref{FigMeanVar}, we show the agreement between Equations
(\ref{EqBbar}) and (\ref{EqVarb}) and the exact mean and variance for
the case of $n=100$.
\begin{figure}[H]
\begin{center}
  \includegraphics[width=3.in]{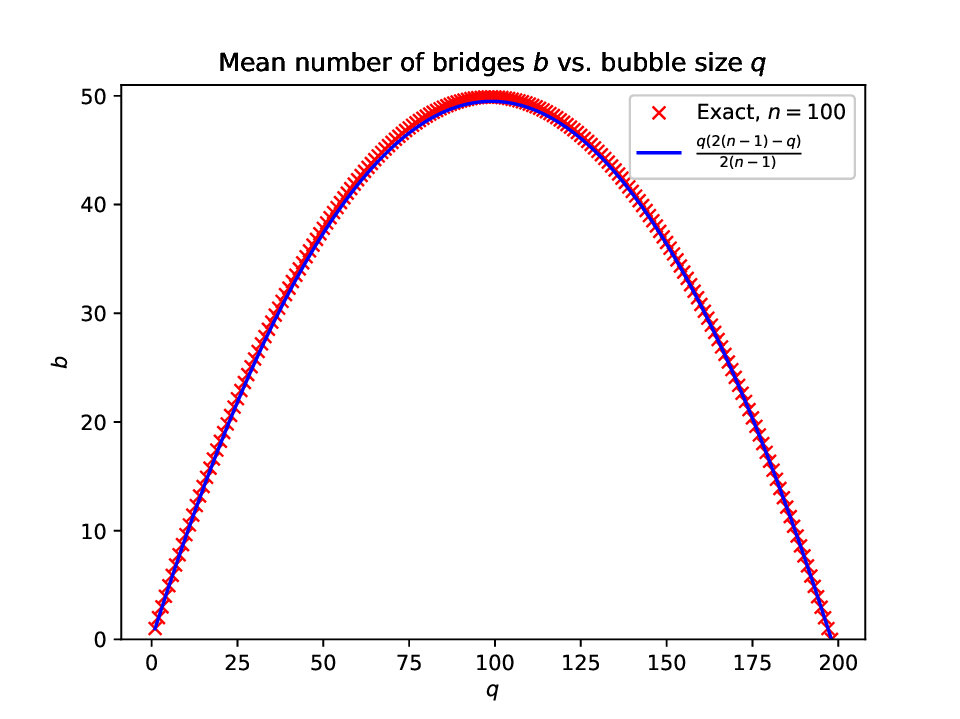}
  \includegraphics[width=3.in]{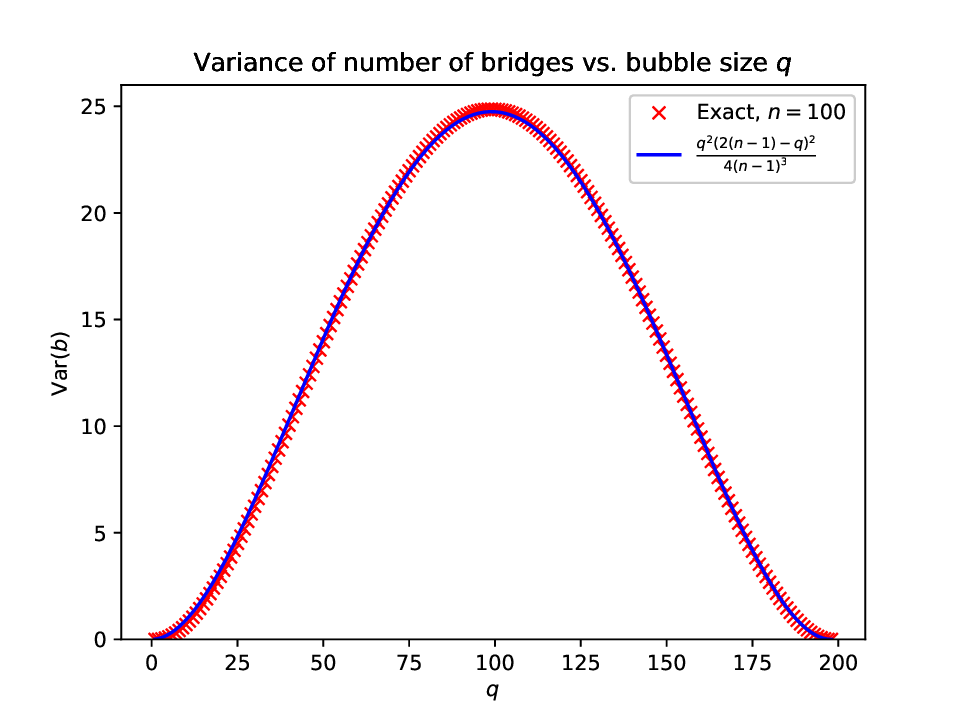}
\end{center}
\caption{On the left: the mean number $\bar b$ of bridges as a
  function of bubble size $q$, counted across all diagrams with
  $n=100$ chords. On the right: the variance of the number of bridges
  as a function of bubble size $q$, counted across all diagrams with
  $n=100$ chords. Exact data in red x's, asymptotic expressions from
  Equations (\ref{EqBbar}) and (\ref{EqVarb}), respectively, in solid
  blue.}
\label{FigMeanVar}
\end{figure}

\section{Bubbles in crystallized diagrams}

In this section we consider the enumeration of chord diagrams
consisting only of {\it empty} bubbles. An empty bubble is one devoid
of internal chords, or, equivalently, every vertex of an empty bubble
is attached to a bridge, see Figure \ref{FigCryst}. When all of a
diagram's bubbles are empty, we refer to the diagram as {\it
  crystallized}.

\begin{figure}[H]
\begin{center}
  \includegraphics[width=3.in]{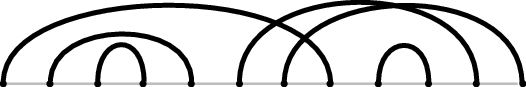}
\end{center}
\caption{An example of a crystallized diagram with $n=6$ chords. This
  diagram has $k=2$ short chords. Note that all three bubbles (of
  sizes $2$, $4$, and $2$ reading from left to right) are devoid of
  internal chords.}
\label{FigCryst}
\end{figure}

\begin{table}[H]
\begin{center}
  \begin{tabular}{c|lllllll}
  $n$ \textbackslash$k$& 1& 2& 3& 4& 5& 6& 7\\
    \hline
1& 1\\
2& 1& 1\\
3& 2& 3& 1\\
4& 6& 12& 6& 1\\
5& 24& 62& 39& 10& 1\\
6& 120& 396& 296& 95& 15& 1\\
7& 720& 3024& 2616& 980& 195& 21& 1\\
\end{tabular}
\end{center}
\caption{The total number, $R_{n,k}$, of crystallized linear chord
  diagrams on $n$ chords containing $k$ short chords. On-line
  Encyclopedia of Integer Sequences \seqnum{A375504}.}\label{TabRnk}
\end{table}

In Table \ref{TabRnk} we show the number, $R_{n,k}$, of crystallized
diagrams consisting of $n$ chords, exactly $k$ of which are short
chords. One can recognize various patterns in the table. In the first
column, we see that $R_{n,1}=(n-1)!$. Crystallized diagrams with a
single short chord consist of two bubbles, each of size $n-1$, with a
central short chord bounding them. Elementary considerations reveal
that there are then $(n-1)!$ distinct ways of joining the vertices in
the two bubbles by chords. Looking at the leading diagonal, it is also
clear that $R_{n,n}=1$; there is only one diagram in which all chords
are short. Finally, on the subleading diagonal we see that $R_{n,n-1}
= T_{n-1}$, where $T_{n-1}=n(n-1)/2$ is the $(n-1)^\text{th}$
triangular number. These diagrams have two vertices which do not
participate in short chords. The first such vertex may be placed in
any of the $n$ gaps between the $n-1$ short chords (including the gaps
formed by the ends of the diagram), whilst the second must be placed
in a different gap, hence giving $n(n-1)/2$ distinct configurations.

\begin{theorem}
  The number of distinct crystallized diagrams consisting of $n$
  chords, exactly $k$ of which are short chords, is given by
\begin{equation}\label{EqRnk}
R_{n,k} = \sum_{\substack{\{p_i\}\\\sum p_i = n-k}}
\frac{w_1!w_2!\cdots w_{k+1}!}{p_1!p_2!\cdots p_{k(k+1)/2}!},
\end{equation}
where $w_i$ is the $i^{\text{th}}$ component of the $(k+1)\times 1$ column
vector ${\bf w}$ given by
\begin{equation}\nonumber
  {\bf w} = {\bf B}\, {\bf p}, 
\end{equation}
where ${\bf B}$ is the incidence matrix of the complete graph on $k+1$
vertices and ${\bf p}$ is the $k(k+1)/2 \times 1$ column vector formed
from the $p_i$.
\end{theorem}
\begin{proof}
The presence of $k$ short chords generically divides a crystallized
diagram into $k+1$ bubbles. A given bubble can therefore be connected,
through its bridges, to any of the $k$ others. It is natural,
therefore, to consider the complete graph on $k+1$ vertices, which we
will denote $K_{k+1}$. We let each vertex of $K_{k+1}$ represent a
bubble, and each edge represent some number $0\leq p_i\leq n-k$, where
$i\in [1,k(k+1)/2]$, of bridges joining the two bubbles in
question. The size of a given bubble is then determined by the sum of
its incident $p_i$: $w_i$ is a sum of $k$ of the
$p_i$'s and represents the size of $i^{\text{th}}$ bubble. It is also
clear that $\sum_i p_i = n-k$, as there are $n-k$ bridges in total.

Were the $n-k$ bridges distinguishable, we would have $w_1!w_2!\cdots
w_{k+1}!$ different ways of arranging the ends of the bridges within
the bubbles. Since they are, in fact, indistinguishable, we must
divide by $p_1!p_2!\cdots p_{k(k+1)/2}!$, yielding the desired result.
\end{proof}

\begin{corollary}
  The total number, $C_{n,k,q}$, of bubbles of size $q$ found amongst all
  crystallized diagrams on $n$ chords, which also contain exactly $k$
  short chords, is given by
\begin{equation}\label{EqCnkq}
C_{n,k,q}=(k+1) \frac{(2n-k-q-1)!}{(2n-k-2q-1)!} 
  \sum_{\substack{\{p_i\}\\\sum p_i = n-k-q}}
\frac{w_1!w_2!\cdots w_k!}{p_1!p_2!\cdots p_{k(k-1)/2}!}.
\end{equation}
\end{corollary}
\begin{proof}
In order to understand this formula, consider the following procedure
on a crystallized diagram of $n$ chords, with $k$ short chords, and at
least one bubble of size $q$. We take one of the bubbles of size $q$
and ``collapse'' it -- we delete the $q$ vertices of the bubble, the
two vertices comprising one of the short chords which bound the bubble
(in the case when the bubble is bounded by the end of the diagram,
there is only one such short chord), and the $q$ vertices found at the
other end of the deleted bubble's bridges. We are left with a
crystallized diagram of $n-q-1$ chords with $k-1$ short chords, and so
$n-q-k$ bridges, see Figure \ref{FigCnkq}.
\begin{figure}[H]
\begin{center}
  \includegraphics[width=4.in]{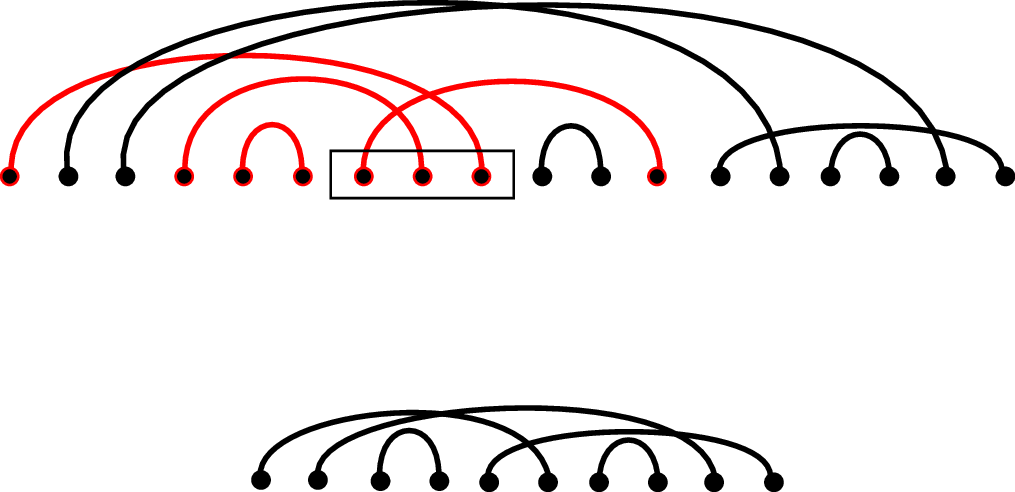}
\end{center}
\caption{Pictured in the top figure is a crystallized diagram with
  $n=9$ chords, $k=3$ short chords, and containing a bubble of size
  $q=3$ which has been indicated by a box. In the collapse procedure
  the elements marked in red are deleted, giving rise to the diagram
  in the bottom figure which has $n-q-1=5$ chords and $k-1=2$ short
  chords.}
\label{FigCnkq}
\end{figure}
\noindent There are $R_{n-q-1,k-1}$ of these diagrams, which
accounts for the summation factor in the expression. We now consider
all the ways in which the deleted bubble could be restored -- by
inserting its $q$ vertices and deleted adjacent short chord at either
end of the diagram or adjacent to any of the $k-1$ short chords --
therefore in any of $2+k-1=k+1$ positions, accounting for the leading
factor of $(k+1)$. The placement of the $q$ ends of the restored
bubble's bridges reduces to a balls-into-bins counting problem. The
$q$ vertices in question are the balls, while the bins are the
positions between the vertices of the $n-q-1$ chords of the collapsed
diagram, not allowing positions inside any of the $k-1$ short chords,
but including the two end positions. This amounts to $2(n-q-1-(k-1)) +
k - 1 +1 = 2n - 2q -k$ bins. As the $q$ balls are distinguishable,
there are $q! {(2n-2q-k)+q-1\choose q}$ placements, which furnishes
the remaining quotient of factorials appearing before the sum.
\end{proof}

\begin{remark}
Comparing Equations (\ref{EqRnk}) and (\ref{EqCnkq}), we note that
\begin{equation}\nonumber
  R_{n,k} = \frac{1}{k+2} C_{n+1,k+1,0},
\end{equation}
which can also be appreciated by noting that ``bubbles of size zero''
come either from adjacent short chords, or short chords at the two
ends of the diagram: given a diagram counted by $R_{n,k}$, one may
place a short chord adjacent to any of the $k$ existing short chords,
or at the two ends of the diagram, and thus at $k+2$ distinct
positions, thus producing bubbles of size zero in diagrams with one
extra (short) chord.
\end{remark}

\subsection{Asymptotic distribution}

We seek to obtain an asymptotic expression for $R_{n,k}$ (and, by
extension, for $C_{n,k,q}$) valid at large values of $n$, $k$, and
$q$. We will therefore concentrate on the sum 
\begin{equation}\nonumber
 \Psi_{N,k} =  \sum_{\substack{\{p_i\}\\\sum p_i = N}}
  \frac{w_1!w_2!\cdots w_{k+1}!}{p_1!p_2!\cdots p_{k(k+1)/2}!},
\end{equation}
as our quantities of interest are then given by
\begin{equation}\nonumber
  R_{n,k} = \Psi_{n-k,k}, \qquad
  C_{n,k,q} = (k+1) \frac{(2n-k-q-1)!}{(2n-k-2q-1)!}\, \Psi_{n-k-q,k-1}.
\end{equation}
At large values of $N$, the sum over the $p_i$ is dominated
by the locus $p_1 = p_2 = \cdots = p_{k(k+1)/2} = \bar n$ where $$\bar
n = \frac{2N}{k(k+1)}.$$ We expand each of the $p_i$ about this
locus as follows:
\begin{equation}\nonumber
  p_i = \bar n + \sqrt{\bar n}\,z_i,\qquad w_i = k\bar
  n+\sqrt{k\bar n}\,u_i, \qquad \sum u_i =  \sum z_i = 0,
\end{equation}
where $\sqrt{k}\,{\bf u} = {\bf B}\,{\bf z}$, and where $u_i$
(respectively $z_i$) is the $i^{\text{th}}$ component of the vector
${\bf u}$ (respectively ${\bf z}$). Using the Stirling approximation,
we find that the summand of $\Psi_{N,k}$, which we will denote
$S\Psi_{N,k}$, may be approximated as follows
\begin{equation}\nonumber
S\Psi_{N,k} =\frac{w_1!w_2!\cdots w_{k+1}!}{p_1!p_2!\cdots p_{k(k+1)/2}!}\simeq
\frac{\prod_{j=1}^{k+1}\left(\frac{k\bar
    n}{e}\left(1+\frac{u_j}{\sqrt{k\bar n}}\right)\right)^{k\bar n
    +\sqrt{k\bar n}u_j} \sqrt{2\pi k\bar n\left(1+\frac{u_j}{\sqrt{k\bar n}}\right)}}
     {\prod_{i=1}^E\left(\frac{\bar n}{e}\left(1+\frac{z_i}{\sqrt{\bar
         n}}\right)\right)^{\bar n +\sqrt{\bar n}z_i} \sqrt{2\pi\bar n \left(1+
         \frac{z_i}{\sqrt{\bar n}}\right)}},
\end{equation}
where $E=k(k+1)/2$. Simplifying, we find 
\begin{equation}\nonumber
S\Psi_{N,k} \simeq(2\pi\bar n)^{(k+1-E)/2}k ^{(k+1)/2}\frac{ \left(\frac{k \bar
    n}{e}\right)^{k(k+1) \bar n+\sqrt{k \bar n}\sum_j
    u_j}}{\left(\frac{\bar n}{e}\right)^{\bar n E +\sqrt{\bar n}\sum_i
    z_i}} \frac{\prod_{j=1}^{k+1}\left(1+\frac{u_j}{\sqrt{k \bar
    n}}\right)^{k \bar n +\sqrt{k \bar n}u_j+1/2} }
{\prod_{i=1}^E\left(1+\frac{z_i}{\sqrt{\bar n}}\right)^{\bar n +\sqrt{\bar
      n}z_i+1/2} }.
\end{equation}
Using the constraints $\sum_j u_j = 0 = \sum_i z_i$, we find
\begin{equation}\nonumber
  \begin{split}
S\Psi_{N,k} \simeq{\cal C}_0 &\left(1+\frac{1}{\sqrt{k \bar n}} \sum_ju_j
+\frac{1}{k \bar
  n}\sum_{j_1<j_2}u_{j_1}u_{j_2}+\cdots\right)^{k \bar n
+1/2}\prod_{j=1}^{k+1}\left(1+\frac{u_j}{\sqrt{k \bar
    n}}\right)^{\sqrt{k \bar n}u_j}\\
\times&\left(1+\frac{1}{\sqrt{\bar n}}\sum_i z_i +\frac{1}{\bar
  n}\sum_{i_1<i_2} z_{i_1}z_{i_2}+\cdots\right)^{-\bar n-1/2}
\prod_{i=1}^E\left(1+\frac{z_i}{\sqrt{\bar n}}\right)^{-\sqrt{\bar
      n}z_i},
\end{split}
\end{equation}
where we have defined the constant:
\begin{equation}\nonumber
{\cal C}_0=(2\pi\bar n)^{(k+1-E)/2}k ^{(k+1)/2}\left(\frac{k \bar
    n}{e}\right)^{k(k+1) \bar n}\left(\frac{\bar n}{e}\right)^{-\bar n
  E}.
\end{equation}
We now use the constraints $\sum_j u_j = 0 = \sum_i z_i$ to write
\begin{equation}\nonumber
  \sum_{i_1<i_2} z_{i_1}z_{i_2} = -\frac{1}{2} \sum_i z_i^2,\qquad
  \sum_{j_1<j_2} u_{j_1}u_{j_2} = -\frac{1}{2} \sum_j u_j^2.
\end{equation}
This allows us to write
\begin{equation}\nonumber
  \begin{split}
&S\Psi_{N,k} \simeq {\cal C}_0 \left(1
-\frac{1}{2k \bar
  n}\sum_{j}u_{j}^2\right)^{k \bar n+1/2
}\exp\left(\sum_j u^2_j\right)
\left(1 -\frac{1}{2\bar
  n}\sum_{i} z_{i}^2\right)^{-\bar n-1/2}
\exp\left(-\sum_i z^2_i\right)
\end{split}
\end{equation}
\begin{equation}\nonumber
  \begin{split}
&\simeq {\cal C}_0
\exp\left(-\sum_i u^2_j/2\right)
\exp\left(\sum_j u^2_j\right)
\exp\left(\sum_i z^2_i/2\right)
\exp\left(-\sum_i z^2_i\right)
\end{split}
\end{equation}
\begin{equation}\nonumber
={\cal C}_0\exp\left(
\frac{1}{2}\sum_{j}u_{j}^2
-\frac{1}{2}\sum_{i} z_{i}^2\right)
\end{equation}
\begin{equation}\nonumber
= {\cal C}_0\exp\left( - \frac{1}{2} \bigl({\bf
  z}^{\mathsf{T}}{\bf z}- {\bf u}^{\mathsf{T}}{\bf u}\bigr)\right)
=  {\cal C}_0\exp\left( - \frac{1}{2} \left({\bf z}^{\mathsf{T}}{\bf z}
-\frac{1}{k} {\bf
  z}^{\mathsf{T}}{\bf B}^{\mathsf{T}}{\bf B}{\bf z}\right)\right)
\end{equation}
We now use the fact that ${\bf B}^\mathsf{T} {\bf B} = 2{\bf I}_E +
{\bf L}$, where ${\bf L}$ is the adjacency matrix of the line graph of
$K_{k+1}$, to obtain
\begin{equation}\nonumber
S\Psi_{N,k} \simeq {\cal C}_0\exp\left(-\frac{1}{2k } {\bf z}^{\mathsf{T}}
\bigl( (k-2){\bf I}_E- {\bf L}\bigr){\bf z}\right).
\end{equation}

We now approximate the sum over the $p_i$ by integrating over the
$z_i$. We first scale the $z_i$ by $1/\sqrt{\bar n}$, so that factors of
$\sqrt{\bar n}$ do not appear in the integration measure. We then
impose the constraint $\sum z_i = 0$ through the introduction of the
Dirac delta function
\begin{equation}\nonumber
  \delta\left(\sum z_i\right) = \frac{1}{2\pi}\int_{-\infty}^\infty d\omega
  \,e^{i\omega\sum z_i}= \frac{1}{2\pi}\int_{-\infty}^\infty d\omega
  \,e^{i\omega{\bf j}_E^\mathsf{T} {\bf z}},
\end{equation}
where ${\bf j}_E$ is a column vector consisting of $k(k+1) /2$
ones. We obtain
\begin{align}\nonumber
  \Psi_{N,k}&\simeq \,{\cal C}_0 \int dz_1 \cdots \int dz_{k(k+1)
    /2}\int\frac{d\omega}{2\pi} \exp\left(-\frac{1}{2\bar nk } {\bf
    z}^{\mathsf{T}} \, {\bf M}\,{\bf z}+ i\omega{\bf j}_E^\mathsf{T}
           {\bf z}\right)\\\label{eqBG} &=  {\cal C}_0\,
           \frac{\left(2\pi k \bar n\right)^{k(k+1)/4-1/2}}
                {\sqrt{{\bf j}_E^\mathsf{T} {\bf M}^{-1}{\bf j}_E\det
                    {\bf M}}},
\end{align}
where ${\bf M} = (k-2){\bf I}_E- {\bf L}$. In order to evaluate the
expression in the denominator, we first establish some basic results
associated with the spectrum of the matrix ${\bf M}$.

\begin{lemma}\label{specA}
The adjacency matrix ${\bf A}$ of $K_{k+1}$ has one eigenvalue equal to $k$
and $k$ others all equal to $-1$.
\end{lemma}
\begin{proof}
We begin by noting that ${\bf A}$ may be written as ${\bf A}= {\bf
  j}_{k+1} {\bf j}_{k+1}^\mathsf{T}-{\bf I}_{k+1}$, where ${\bf
  j}_{k+1}$ is a $(k+1)\times 1$ column vector of $1$'s. The vector
${\bf j}_{k+1}$ is an eigenvector of the matrix ${\bf j}_{k+1} {\bf
  j}_{k+1}^\mathsf{T}$ since ${\bf j}_{k+1} {\bf j}_{k+1}^\mathsf{T}
{\bf j}_{k+1} = (k+1){\bf j}_{k+1} $. This also tells us that $k+1$ is
one eigenvalue of ${\bf j}_{k+1} {\bf j}_{k+1}^\mathsf{T}$, but since this
matrix is formed of $k+1$ copies of the same column, it has rank $1$ and
all of its $k$ other eigenvalues are therefore zero. The eigenvalues
of ${\bf A}$ are those of ${\bf j}_{k+1} {\bf j}_{k+1}^\mathsf{T}$ shifted by
$-1$, and so the result follows.
\end{proof}
\begin{lemma}\label{BBT}
The matrix ${\bf B}{\bf B}^\mathsf{T}$, where ${\bf B}$ is the
incidence matrix of $K_{k+1}$, has one eigenvalue equal to $2k $ and
$k$ others all equal to $k-1$.
\end{lemma}
\begin{proof}
We use the result ${\bf B}{\bf B}^\mathsf{T} = k {\bf I}_{k+1} + {\bf
  A}$. This implies that the eigenvalues of ${\bf B}{\bf
  B}^\mathsf{T}$ are those of ${\bf A}$, with each eigenvalue
increased by $k$. Using the results of Lemma \ref{specA}, we obtain
the desired result.
\end{proof}
\begin{lemma}\label{BTB}
The adjacency matrix ${\bf L}$ of the line graph of $K_{k+1}$ has one
eigenvalue equal to $2(k-1)$, $k$ others all equal to $k-3$, and a
final set of $(k+1)(k-2)/2$ eigenvalues all equal to -2.
\end{lemma}
\begin{proof}
We use the result ${\bf B}^\mathsf{T} {\bf B}= 2{\bf I}_{E} + {\bf
  L}$, where $E=k(k+1) /2$ is the number of edges in $K_{k+1}$. The
eigenvalues of ${\bf B}^\mathsf{T} {\bf B}$ are intimately connected
to those of ${\bf B}{\bf B}^\mathsf{T}$. To see this we construct two
block matrices
\begin{equation}\nonumber
  S=\begin{pmatrix}
    &\lambda{\bf I}_{E} &\sqrt{\lambda}{\bf B}^\mathsf{T} \\
    &\sqrt{\lambda}{\bf B} &\lambda{\bf I}_{k+1}
  \end{pmatrix},\quad
  T=\begin{pmatrix}
    &\lambda{\bf I}_{k+1} &\sqrt{\lambda}{\bf B} \\
    &\sqrt{\lambda}{\bf B}^\mathsf{T} &\lambda{\bf I}_{E}.
  \end{pmatrix}
\end{equation}
We note that there is a permutation $p\in S_{k+E}$ which, when applied
to both the rows and to the columns of $S$, gives $T$. Therefore $\det
S = \det T$, or, using the standard result for the determinant of
block matrices,
\begin{equation}\nonumber
  \det(\lambda {\bf I}_{k+1})\det\left(\lambda {\bf I}_E
  - \sqrt{\lambda}{\bf B}^\mathsf{T}
  \left(\lambda {\bf I}_{k+1}\right)^{-1}
  \sqrt{\lambda}{\bf B}\right)=
  \det(\lambda {\bf I}_E)\det\left(\lambda {\bf I}_{k+1}
  - \sqrt{\lambda}{\bf B}
  \left(\lambda {\bf I}_E\right)^{-1}
  \sqrt{\lambda}{\bf B}^\mathsf{T}\right).
\end{equation}
We therefore find that
\begin{equation}\nonumber
  \lambda^{k+1}\det\left(\lambda {\bf I}_E
  - {\bf B}^\mathsf{T}
 {\bf B}\right)=
  \lambda^E\det\left(\lambda {\bf I}_{k+1}
  - {\bf B}
 {\bf B}^\mathsf{T}\right),
\end{equation}
and so
\begin{equation}\nonumber
  \det\left(\lambda {\bf I}_E
  - {\bf B}^\mathsf{T}
 {\bf B}\right)=
  \lambda^{E-k-1}\det\left(\lambda {\bf I}_{k+1}
  - {\bf B}
  {\bf B}^\mathsf{T}\right).
\end{equation}
The conclusion is that the spectrum of ${\bf B}^\mathsf{T}{\bf B}$ is
that of ${\bf B}{\bf B}^\mathsf{T}$, with the addition of $E-k-1$ zero
eigenvalues. Since ${\bf L}={\bf B}^\mathsf{T} {\bf B}-2{\bf I}_{E}$,
the eigenvalues of ${\bf L}$ are those of ${\bf B}^\mathsf{T} {\bf B}$
minus 2, and in light of the results of Lemma \ref{BBT}, we obtain the
desired result.
\end{proof}
\begin{lemma}\label{lemgs}
The grand sum of the inverse, ${\bf j}_E^\mathsf{T} {\bf M}^{-1}{\bf
  j}_E$, is equal to $-(k+1)/2$.
\end{lemma}
\begin{proof}
We note that ${\bf j}_E$ is an eigenvector of ${\bf L}$, with
eigenvalue equal to the number of edges any given edge in $K_{k+1}$ is
adjacent to, which is $2(k-1)$. Therefore, ${\bf M} \,{\bf j}_E =
((k-2)-2(k-1)) {\bf j}_E = -k {\bf j}_E$. We therefore have that
${\bf M}^{-1}{\bf j}_E =-{\bf j}_E/k  $, and so ${\bf
  j}_E^\mathsf{T}{\bf M}^{-1}{\bf j}_E =-E/k  =-(k+1)/2$.
\end{proof}
\begin{theorem}\label{thmeig}
The matrix ${\bf M}= (k-2){\bf I}_E- {\bf L}$ has eigenvalues $-k $,
$1$, and $k$, with multiplicities $1$, $k$, and $(k+1)(k-2)/2$
respectively.
\end{theorem}
\begin{proof}
We note that the eigenvalues of ${\bf M}$ are minus those of ${\bf L}$
plus $k-2$. Using the results of Lemma \ref{BTB}, we obtain the desired result.
\end{proof}
\begin{corollary}\label{CorBG}
\begin{equation}\nonumber
  {\bf j}_E^\mathsf{T}
  {\bf M}^{-1}{\bf j}_E\det {\bf M} = \frac{(k+1)k^{k(k-1)/2}}{2}.
\end{equation}
\end{corollary}
\begin{proof}
  This follows directly from Lemma \ref{lemgs} and Theorem \ref{thmeig}.
\end{proof}

With this result in hand, we may return to evaluating $R_{n,k}$ in the
limit of large diagrams.
\begin{theorem}\label{ThmRasym}
In the $n\to \infty$ limit, the number, $R_{n,k}$, of crystallized
diagrams with $k$ short chords, is given by
\begin{align}\nonumber
  R_{n,k}\simeq (k+1)^{-1/2}\,2^{n+1/2}\, e^{k-n}\,
  k^{n-k/2} \,\pi^{k/2}\,
  \left(\frac{n-k}{k+1}\right)^{n-k/2}.
\end{align}
\end{theorem}
\begin{proof}
Using the fact that $R_{n,k}=\Psi_{n-k,k}$, and using Equation
(\ref{eqBG}), in conjunction with the result of Corollary \ref{CorBG},
the result follows.
\end{proof}

We will now show that $R_{n,k}$ is tightly peaked about a mean, whose
value we shall determine.
\begin{theorem}
  In the $n\to\infty$ limit, the distribution of short chords in
  crystallized diagrams is normal with mean given by
  \begin{equation}\label{kbar}
 \bar k \simeq \sqrt{\frac{2n}{\log n}}.
\end{equation}
\end{theorem}
\begin{proof}  
Using the result of Theorem \ref{ThmRasym}, and keeping only leading
contributions in the limit of large $n$ and $k$, we find
\begin{equation}\label{expR}
  \frac{\partial \log R_{n,k}}{\partial k} \simeq \frac{n}{k^2}
  -\frac{1}{2}\log n.
\end{equation}
Setting this to zero and solving for $k$, we find the mean value of
$k$ given in Equation (\ref{kbar}). The variance of the distribution
is then obtained as follows
\begin{equation}\nonumber
  \text{Var}(k)\simeq-\left(\frac{\partial^2 \log R_{n,\bar k}}{\partial k^2} \right)^{-1}
  \simeq \frac{{\bar k}^3}{2n} \simeq \sqrt{\frac{2n}{\log^3 n}}.
\end{equation}
As a proportion of the full range of possible values for $k$,
i.e. $k\in[1,n]$, both $\bar k$ and $\text{Var}(k)$ tend to zero as
$n\to\infty$. The distribution therefore consists of a sharp peak near
zero, see Figure \ref{FigRs}. To establish normality of $k$, we note that
\begin{equation}\nonumber
  \frac{\partial^p \log R_{n,\bar k}}{\partial k^p} =
       {\cal O}\left( n^{(1-p)/2}\log^{(p+1)/2}n \right),
\end{equation}
while
\begin{equation}\nonumber
\left(\text{Var}(k)\right)^{-p/2} = {\cal O}\left( n^{-p/4}\log^{3p/4}n \right),
\end{equation}
and so decreases slower, which implies that the higher standardized
moments of $k$ are suppressed as $n\to\infty$.
\end{proof}

In practice, it is difficult to calculate $R_{n,k}$ for values of $n$
large enough to make the asymptotic expression in Theorem
\ref{ThmRasym} a good approximation: the Gaussian integration we
performed was contingent on $2N/(k(k+1)) = 2(n-k)/(k(k+1))$ being
large. Given our expression for $\bar k$, this implies that $\log n
\gg 1$, which implies that very large values of $n$ are
required. Subleading corrections to $\bar k$ may be obtained by
expanding Equation (\ref{expR}) to the next order
\begin{equation}\nonumber
  \frac{\partial \log R_{n,k}}{\partial k} \simeq \frac{n}{k^2}
  -\frac{n}{k^3} -\frac{1}{2k}-\frac{1}{2}\log \left(\frac{n}{\pi}\right)=0.
\end{equation}
Solving this cubic equation in $k$ and then taking the large-$n$ limit, one obtains
\begin{equation}\label{expRc}
\bar k \simeq \sqrt{\frac{2n}{\log (n/\pi)}} -\frac{1}{2} -\frac{1}{2\log(n/\pi)}.
\end{equation}
In Figure \ref{FigKmean}, we show the agreement between this
expression and the actual mean value of $k$.
\begin{figure}[H]
\begin{center}
  \includegraphics[width=3.5in]{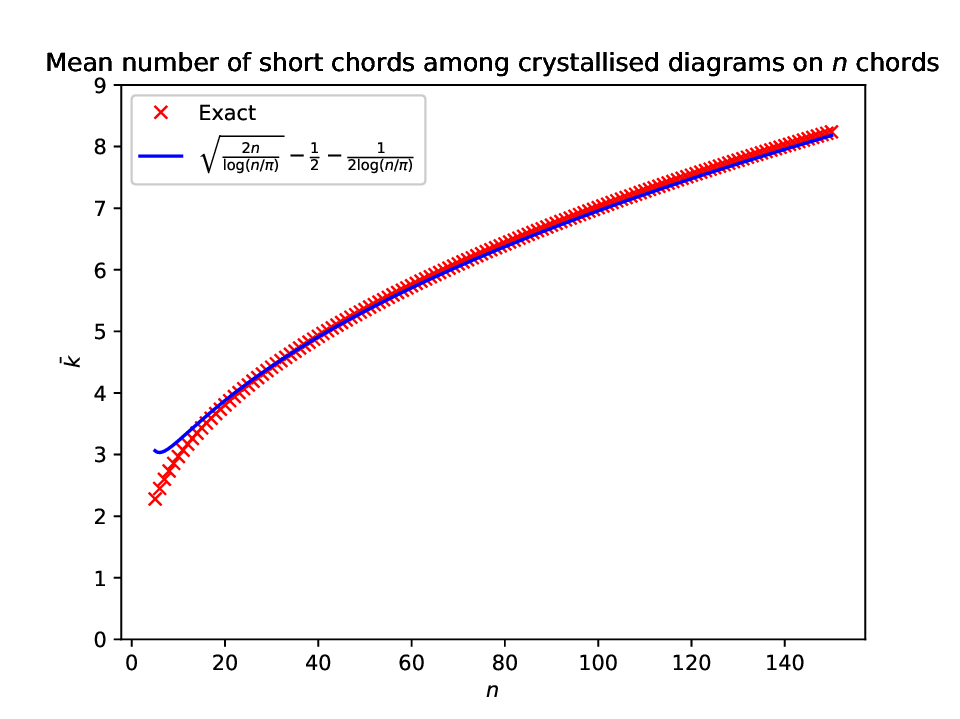}
\end{center}
\caption{The mean value of the number of short chords $\bar k$,
  counted across all crystallized diagrams consisting of $n$ chords,
  for various values of $n$. The blue line is the asymptotic
  expression given in Equation (\ref{expRc}). }
\label{FigKmean}
\end{figure}
\begin{figure}[H]
\begin{center}
  \includegraphics[width=3.1in]{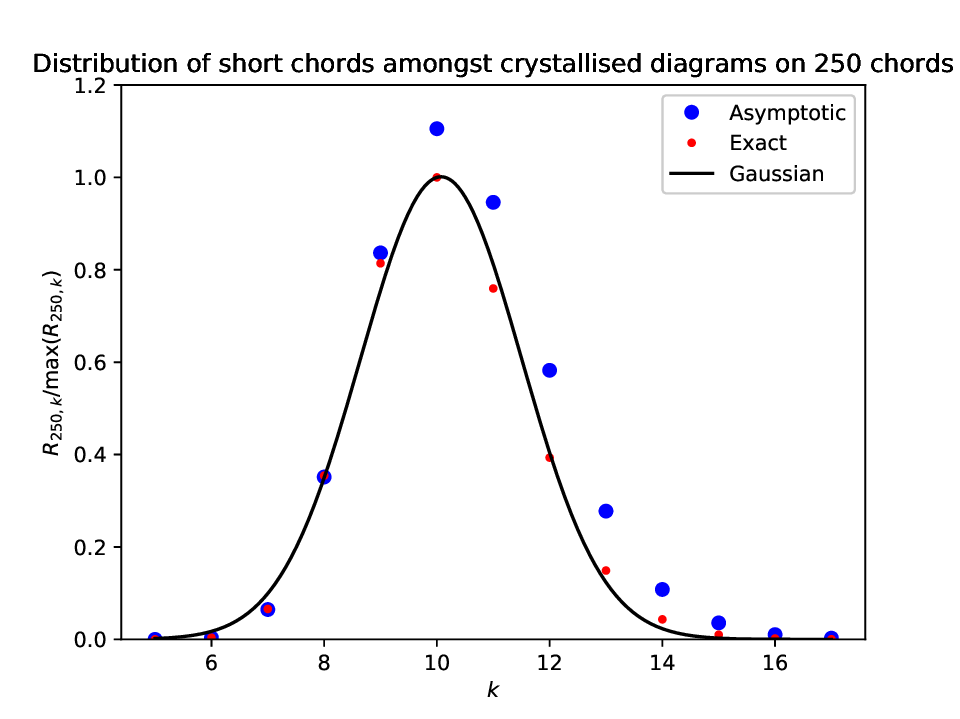}
  \includegraphics[width=3.1in]{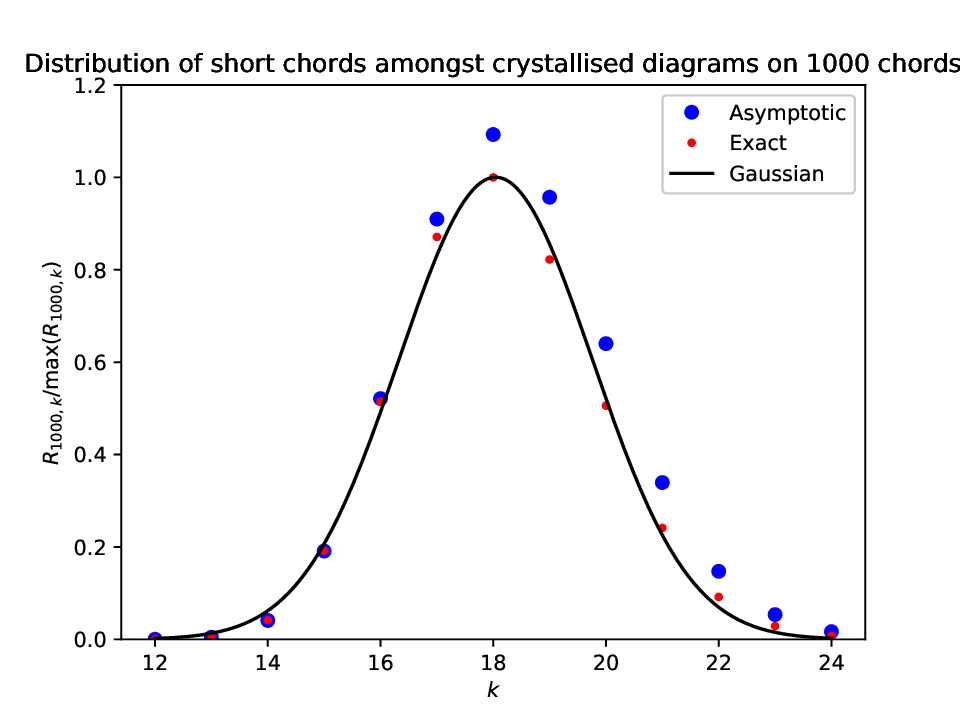}
\end{center}
\caption{The distribution of short chords in crystallized diagrams
  built from $n=250$ chords (left) and $n=1000$ chords (right). In red
  is the exact data, normalized by the largest value
  $\max_k(R_{n,k})$. In blue the asymptotic expression from Theorem
  \ref{ThmRasym} is shown, normalized by the same value. In black a
  Normal curve with mean $\bar k$ given by Equation (\ref{expRc}), and
  variance given by $\bar k^3/(2n)$ has been plotted and normalized in
  such a way that it passes through $\max_k(R_{n,k})$. One can see
  that the exact data is close to Normal, and that the asymptotic
  expression is slowly converging to it (the blue dots are closer to
  the red dots in the figure on the right).}
\label{FigRs}
\end{figure}
\noindent Although the distribution of short chords is already quite close to a
Normal distribution even for modest values of $n\simeq 250$, the
asymptotic form of the distribution given in Theorem \ref{ThmRasym} is
slow to converge to it, see Figure \ref{FigRs}.

Turning our attention to the distribution of bubble sizes, and using the fact that
\begin{equation}\nonumber
  C_{n,k,q} = (k+1) \frac{(2n-k-q-1)!}{(2n-k-2q-1)!}\, \Psi_{n-k-q,k-1},
\end{equation}
we use the Stirling approximation on the factorials to achieve an
expression valid for $n \gg k \simeq q \gg 1$:
\begin{align}\nonumber
  C_{n,k,q}\simeq k^{-1/2}\,&2^{n-q-1/2}\, e^{k-n}\,
  (k-1)^{n-q-(k+1)/2} \,(k+1)\,\pi^{(k-1)/2}\,
  \left(\frac{n-k-q}{k}\right)^{n-q-(k+1)/2}\\\nonumber
  & (2n-2q-k-1)^{1/2+k-2n+2q}\,(2n-q-k-1)^{-1/2-k+2n-q}.
\end{align}
In order to find the mean values of $k$ and $q$, we consider the
equations
\begin{equation}\nonumber
  \frac{\partial \log C_{n,k,q}}{\partial k} \simeq \frac{n}{k^2}
  -\frac{1}{2}\log n = 0, \qquad
  \frac{\partial \log C_{n,k,q}}{\partial q} \simeq \frac{1}{k}
  -\frac{q}{2n} = 0,
\end{equation}
whose solutions give the same value of $\bar k$ as we found in
Equation (\ref{kbar}), and also
\begin{equation}\nonumber
\bar q \simeq \frac{2n}{\bar k} \simeq \sqrt{2n\log n}. 
\end{equation}
This makes sense -- we expect the $\bar k$ short chords to give rise
to $\bar k+1\simeq \bar k$ bubbles, each of size $\bar q\simeq 2n/\bar
k$, on average.

\section{Open questions}

In order to form crystallized diagrams from general linear chord
diagrams, one could consider a crystallization process as follows:
\begin{enumerate}
\item At each time step an endpoint $i$ of a non-short chord is chosen
  uniformly at random.

\item The neighbouring endpoint $j=i-1$ or $j=i+1$ is also
chosen at random (if $i=1$ or $2n$ there is a unique choice of
neighbour).

\item If $j$ is not the endpoint of a short chord, its position is
  swapped with $i$.

\item The process is stopped when all bubbles are empty.
\end{enumerate}
\noindent There are two immediate questions which arise from
considering such a process:

\begin{question}
If the process is applied to general linear chord diagrams chosen
uniformly, does it generate crystallized diagrams with the
same asymptotic distribution of short chords found in Theorem
\ref{ThmRasym}?
\end{question}

\begin{question}
What is the asymptotic distribution of stopping times for the
crystallization process, taken across all linear chord diagrams
uniformly?
\end{question}

The concept of a short chord extends readily to graphs other than the
path of length $2n$, cf.\ Young \cite{DY1}\cite{DY2}, and it would be
interesting to count bubbles in other graphs, especially in the
rectangular grid graphs. There, a bubble would be defined as a
two-dimensional region surrounded short chords, and a generalised
crystallization process could also be considered.

Finally, we note the appearance of an interesting coincidence:

\begin{question}
In the $n\to\infty$ limit, the average linear crystallized diagram
consists of $\bar k \simeq \sqrt{2n/\log n}$ bubbles, and therefore
$\simeq \bar k^2/2$ ``channels'' of bridges connecting them, with each
channel consisting of $\simeq \log n$ bridges joining the same two
bubbles. It is curious to note that this number of channels is
coincident with the leading asymptotics of the prime counting
function, i.e. $\bar k^2/2 \simeq n/\log n$. It would be nice to
understand if there is a connection here to the distribution of the
primes.
\end{question}

\bigskip
\hrule
\bigskip

\noindent {\it 2010 Mathematics Subject Classification:} 
Primary 05A15; Secondary 05C70, 60C05. 

\noindent \emph{Keywords:} 
chord diagram, perfect matching. 
\bigskip
\hrule
\bigskip

\noindent (Concerned with sequences \seqnum{A079267},
\seqnum{A278990}, \seqnum{A367000}, \seqnum{A375504},
\seqnum{A375505}.)


\begin{thebibliography}{99}

\bibitem{CK} N. T. Cameron and K. Killpatrick, Statistics on linear chord
  diagrams, {\it Discrete Math. Theor. Comput. Sci.}
  {\bf 21} (2019) \#11.   Available at
  \url{https://dmtcs.episciences.org/6038/pdf}.

\bibitem{KO} E. Krasko and A. Omelchenko, Enumeration of chord
  diagrams without loops and parallel chords, {\it
    Electron. J. Combin.} {\bf 24} (2017),
  \href{https://www.combinatorics.org/ojs/index.php/eljc/article/view/v24i3p43}{Article
    P3.43}.

\bibitem{KP} G. Kreweras and Y. Poupard, Sur les partitions en paires
  d'un ensemble fini totalement ordonn\'{e}, {\it Publications de
    l'Institut de Statistique de l'Universit\'{e} de Paris} {\bf 23}
  (1978), 57--74.

\bibitem{PR} V. Pilaud and J. Ru\'{e}, Analytic combinatorics of chord
  and hyperchord diagrams with $k$ crossings, {\it Adv. Appl. Math.}
  {\bf 57} (2014), 60--100  

\bibitem{R} J. Riordan, The distribution of crossings of chords
  joining pairs of $2n$ points on a circle, {\it Math. Comp.} {\bf
    29} (1975), 215--222.

\bibitem{Sloane} N. J. A. Sloane et al., The on-line encyclopedia of
integer sequences, 2020.  Available at
\url{https://oeis.org}.
  
\bibitem{T} J. Touchard, Sur un probl\`{e}me de configurations et sur
  les fractions continues, {\it Canad. J. Math.} {\bf 4} (1952), 2--25.
  
\bibitem{DY4} D. Young, Counting bubbles in linear chord diagrams,
    {\it J. Integer Sequences} {\bf 24} (2024),
    \href{https://cs.uwaterloo.ca/journals/JIS/VOL27/Young/young16.html}{Article
      24.4.2}.

\bibitem{DY1} D. Young, The number of domino matchings in the game of
  memory, {\it J. Integer Sequences} {\bf 21} (2018),
  \href{https://cs.uwaterloo.ca/journals/JIS/VOL21/Young/young2.html}{Article
    18.8.1}.

\bibitem{DY2} D. Young, Generating functions for domino matchings in
  the $2 \times k$ game of memory, {\it J. Integer Sequences} {\bf 22}
  (2019),
  \href{https://cs.uwaterloo.ca/journals/JIS/VOL22/Young/young13.html}{Article
    19.8.7}.
  
\end{thebibliography}
\end{document}